\documentclass[11pt,reqno]{amsart} 
\usepackage[height=190mm,width=130mm]{geometry}
\usepackage{amsmath}
\usepackage{amssymb,latexsym}
\usepackage{amsrefs}
\usepackage[draft=true]{hyperref}
\usepackage{cite}
\usepackage{a4wide}
\usepackage{amscd}
\usepackage{graphics}
\usepackage{color}
\usepackage[all]{xy}
\usepackage[utf8]{inputenc}
\usepackage[T1]{fontenc}
\usepackage{multirow}
\usepackage{array}
\usepackage{booktabs,caption}

\theoremstyle{plain}
\newtheorem{thm}{Theorem}
\newtheorem{lemma}{Lemma}
\newtheorem{corollary}{Corollary}
\newtheorem{prop}{Proposition}
\newtheorem{example}{Example}

\theoremstyle{definition}
\newtheorem{defn}{Definition}

\theoremstyle{remark}
\newtheorem{remark}{Remark}

\newcommand{\Z}{\ensuremath{\mathbb{Z}}}   
\newcommand{\N}{\ensuremath{\mathbb{N}}}
\newcommand{\Q}{\ensuremath{\mathbb{Q}}}

\newcommand{\Ker}{\operatorname{{ \it ker}}}
\newcommand{\Rad}{\operatorname{{\it rad}}}
\newcommand{\Soc}{\operatorname{{\it soc}}}
\newcommand{\ds}{\subseteq ^{\oplus}}

\numberwithin{equation}{section} 
\begin{document}

\title[virtually regular modules]{VIRTUALLY REGULAR MODULES}

\author{ENG\.{I}N B\"uy\"uka\c{s}{\i}k}

\address{Izmir Institute of Technology \\ Department of Mathematics\\ 35430 \\ Urla, \.{I}zmir\\ Turkey}

\email{enginbuyukasik@iyte.edu.tr}

\author{\"{O}zlem Irmak Demir}

\address{Izmir Institute of Technology \\ Department of Mathematics\\ 35430 \\ Urla, \.{I}zmir\\ Turkey}
\email{ozlemirmak@iyte.edu.tr}

\begin{abstract}   We call a right module $M$ \textit{(strongly) virtually regular} if every (finitely generated) cyclic submodule is isomorphic to a direct summand.  $M$ is said to be \textit{completely virtually regular} if every submodule is virtually regular.  In this paper, characterizations and some closure properties of the aforementioned modules are given.  Several structure results are obtained over commutative rings. In particular, the structures of finitely presented (strongly) virtually regular modules and completely virtually regular modules  are fully determined over valuation domains.  Namely, for a valuation domain $R$ with the unique maximal ideal $P$, we show that finitely presented (strongly) virtually regular modules are free if and only if $P$ is not principal; and that $P=Rp$ is principal if and only if finitely presented virtually regular modules are of the form $$R^n \oplus (\frac{R}{Rp})^{n_1} \oplus (\frac{R}{Rp^2})^{n_2} \oplus \cdots \oplus (\frac{R}{Rp^k})^{n_k}$$ for nonnegative integers 
$n,\,k,\,n_1,\,n_2,\cdots ,n_k.$  
Similarly, we prove that $P=Rp$ is principal if and only if  finitely presented   strongly virtually regular modules are of the form  $ R^n  \oplus (\frac{R}{Rp})^{m}$, where $m,n$ are nonnegative integers.

 We also obtain that, $R$ admits   a  nonzero finitely presented completely virtually regular module $M$ if and only if $P=Rp$ is principal. Moreover, for a finitely presented $R$-module $M$, we prove that: $(i)$ if $R$ is not a DVR,  then $M$ is completely virtually regular if and only if  $M \cong (\frac{R}{Rp})^{m}$; and $(ii)$ if $R$ is a DVR, then $M$ is completely virtually regular if and only if $M\cong R^n \oplus (\frac{R}{Rp})^{m}.$ Finally, we obtain a characterization of finitely generated virtually regular modules over the ring of integers. 

\end{abstract}

\subjclass[2010]{16D50, 16D60, 18G25}

\keywords{Regular rings; strongly regular modules; virtually regular modules; valuation domains.}

\maketitle

\section{Introduction}

Von Neumann Regular rings and their generalizations to modules are extensively investigated in the literature.  A ring $R$   is said to be \textit{(von Neumann) regular} if for each $a \in R$, there is  $b \in R$ such that $a=aba$.   A well-known characterization of regular rings is that, each principal left or right ideal is generated by an idempotent. That is, $R$ is regular if and only if every principal left and right ideal is a direct summand.  Generalizations of regularity to modules are studied by many authors (see, \cite{Nicholson}, \cite{Ramamurti:Onfinitelyinjectivemodules}, \cite{ware} and \cite{zelmanov}). In \cite{ware}, regularity  extended to projective modules; namely, a projective right module is said to be \textit{regular} if every cyclic submodule is a direct summand.  
More generally, in \cite{Ramamurti:Onfinitelyinjectivemodules}, a right $R$-module $M$ is said to be \textit{strongly regular} if every cyclic submodule of $M$ is a direct summand. It is easy to see that, $M$ is strongly regular if and only if every finitely generated submodule is a direct summand; and that submodules of strongly regular modules are also strongly regular (see, Theorem \ref{strreg}). 

Recently, many well-known and significant structures of rings and modules, such as semisimplicity, and being uniserial, have been studied from a different perspective.  In \cite{vss}, a right $R$-module $M$ is said to be \textit{virtually semisimple} if every submodule of $M$ is isomorphic to a direct summand of $M$.   The authors obtained a generalization of the Wedderburn-Artin Theorem by  using virtually semisimplicity (see, \cite{vss}, \cite{wa}). Accordingly, structure of virtually semisimple modules over
commutative rings examined in \cite{svss}.  By a similar approach, virtually uniserial modules are investigated and  studied in \cite{vuniserial}. These new approach and generalizations lead to new and interesting characterizations of some well-known rings in terms of these concepts.

  


Following these recent new ideas, in this paper,  we investigate virtually regular, completely virtually regular and strongly virtually regular modules which arose from equivalent characterizations of strongly regular modules (see Theorem \ref{strreg}). Namely, we call a right module $M$ \textit{(strongly) virtually regular} if every cyclic (finitely generated) submodule of $M$ is isomorphic to a direct summand of $M$.  $M$ is said be \textit{completely virtually regular} if every submodule of $M$ is virtually regular. 
We have the following hierarchy and implications for the aforementioned modules: 


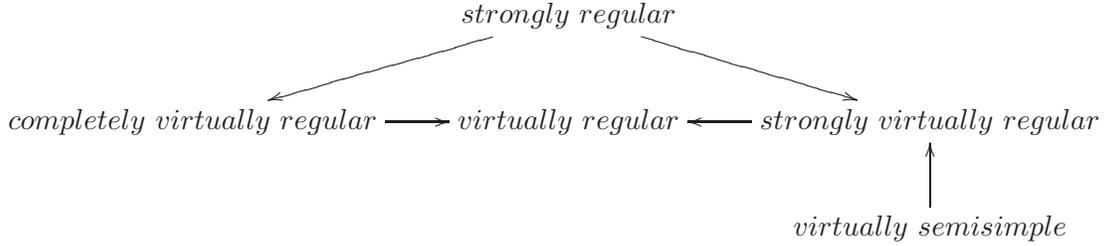
\begin{center}
	\begin{displaymath}
	\xymatrix{& strongly\,\, regular \ar[dl]  \ar[dr]  & \\
		completely\,\,virtually\,\,regular \ar[r]  &  virtually\,\, regular   & strongly \,\,virtually\,\, regular \ar[l] \\
		&   & virtually\,\,semisimple   \ar[u] }  
	\end{displaymath}
\end{center}

We prove several module theoretic characterizations of (strongly) virtually regular and completely virtually regular modules. Many examples are provided to distinguish these classes of modules from each other and several related concepts.  We obtain some structure theorems for virtually regular and completely virtually regular modules, in particular over valuation domains and over the ring of integers.

The paper is organized as follows.

In section 2, we provide examples to distinguish the notions of virtually semisimple, virtually regular, completely virtually regular and strongly virtually regular modules.  For a ring $R$, we prove that $R$ is right virtually regular if and only if $R$ is right pp-ring, and $R$ is strongly virtually regular if and only if $R$ is right  pp-ring and right B\'{e}zout ring. A domain $R$ is right completely virtually regular if and only if $R$ is a right PID. All right $R$-modules are virtually regular if and only if injective right $R$-modules are virtually regular if and only if $R$ is semisimple Artinian. The rings whose cyclic right modules are virtually regular are right $V$-rings (i.e. simple right modules are injective). A right PCI domain is a nonsemisimple ring whose cyclic right modules are virtually regular.

In section 3, we study virtually regular modules over commutative rings. We show that, a torsion-free module $M$ over a commutative domain $R$ is virtually regular if and only if $M\cong N\oplus R$ for some $R$-module $N$ (Proposition \ref{torsion-freecvssoverdomain}).  If $M$ is virtually regular over a commutative domain $R$, then $T(M)$ and $\frac{M}{T(M)}$ are virtually regular, and the converse implication holds if $R$ is almost maximal Prüfer (Theorem \ref{prop:torsiontorsionfreepart}).  An indecomposable $R$-module $A$ is virtually regular if and only if $A\cong \frac{R}{P}$ for some prime ideal $P$ of $R$. An indecomposable  $R$-module $B$ is completely virtually regular if and only if $B \cong \frac{R}{Q}$ is a PID.   A torsion module over a Dedekind domain is completely virtually regular if and only if it is semisimple (Proposition \ref{prop:torsionoverdedekinddomain}). As a byproduct, the structure of completely virtually regular modules over Dedekind domain is obtained (Corollary \ref{cor:cvroverDD}).

The focus of our study in section 4 is on the structure of finitely presented (strongly) virtually regular and completely virtually regular modules over valuation domains. We obtain  complete characterization of finitely presented  (strongly) virtually regular modules and completely virtually regular modules over valuation domains.  Let $R$ be valuation domain with the unique maximal ideal $P$. We prove that,  finitely presented (strongly) virtually regular modules are free if and only if $P$ is not principal.
If $P=Rp$ is principal, then a finitely presented $R$-module $M$ is  virtually regular if and only if

 $$M\cong R^n \oplus (\frac{R}{Rp})^{n_1} \oplus (\frac{R}{Rp^2})^{n_2} \oplus \cdots \oplus (\frac{R}{Rp^k})^{n_k}$$for nonnegative integers $n,\,k,\,n_1,\,n_2,\cdots ,n_k$ (Theorem \ref{thm:fpvirtuallyregularovervd}). Similarly, $R$ admits a non free finitely presented strongly virtually regular module $M$ if and only if $P=Rp$ is principal and  $M\cong R^n \oplus (\frac{R}{Rp})^{m}$ for nonnegative integers $n,\,m$ not both zero (Proposition \ref{prop:svrovervd}).

The case for finitely presented completely virtually regular modules is slightly different over valuation domains.  
We prove that, $R$ admits a (nonzero) finitely presented completely virtually regular module if and only if $P=Rp$ is principal.  Moreover, for a finitely presented $R$-module $M$, we prove that: $(i)$ if $R$ is not a DVR,  then $M$ is completely virtually regular if and only if  $M \cong (\frac{R}{Rp})^{m}$; and $(ii)$ if $R$ is a DVR, then $M$ is completely virtually regular if and only if $M\cong R^n \oplus (\frac{R}{Rp})^{m},$ where $n,\,m$ are nonnegative integers (Proposition \ref{prop:cvrovervaluationdomains}).

In section 5, we characterize finitely generated virtually regular and completely virtually regular modules over the ring of integers (Theorem \ref{cvfag}).  The notions virtually semisimple, completely virtually regular and strongly virtually regular coincide for finitely generated abelian groups (Corollary \ref{corollary:vss-cvr-svr}).

Throughout this paper, $R$ is a ring with unity and modules are unital right $R$-modules. As usual, we denote the radical, the socle and the singular submodule of a module $M$ by $\Rad(M)$, $\Soc(M)$ and $Z(M)$, respectively.   We write $N \subseteq M$ if $N$ is a submodule of $M$,  $N \ds M$   if $N$ is a direct summand of $M$.  We also write $M^{(I)}$ for a direct sum of $I$-copies of $M$. The torsion submodule of a module $M$ over a commutative domain is denoted by $T(M)$. We refer to \cite{AF}, \cite{FuchsAndSalce:ModulesOverNonNoetherianDomains} and \cite{Lam:LecturesOnModulesAndRings}  for all the undefined notions in this paper.

\section{Virtually Regular modules}

Recall that a ring $R$ is regular if and only if every principal right ideal of $R$ is a direct summand. Ramamurthi et al. in \cite[page 246]{Ramamurti:Onfinitelyinjectivemodules} generalized this notion to modules and investigated strongly regular modules. An $R$-module $M$ is said to be \textit{strongly regular} if every finitely generated submodule of $M$ is a direct summand.

We recall the following known characterization of strongly regular modules for completeness. 

\begin{thm}\label{strreg}
Let $M$ be an $R$-module. The following statements are equivalent.
\begin{enumerate}
  \item $M$ is strongly regular.
  \item Every cyclic submodule of $M$ is a direct summand.
  \item For every submodule $N \subseteq M$, every cyclic submodule of $N$ is a direct summand of $N.$
  \item Every finitely generated submodule of $M$ is a direct summand.
\end{enumerate}
\end{thm}
\begin{proof}
 $ (1) \Rightarrow (2)$, $(3) \Rightarrow (2)$ and $(1) \iff (4)$ are clear. 
 
$ (2) \Rightarrow (3)$ Let $N$ be a submodule of $M$, and $C$ be a cyclic submodule of $N$. Then $C\ds M$  by $(2)$, and so $C \ds N$ by the modular law.

(2) $\Rightarrow$ (4)  Let $A$ be a submodule of $M$ which is generated by $a_1, a_2, \cdots, a_n$. By induction on $n$, we will show that $A \ds M$. If $n = 1$, then $A$ is cyclic, and so  $ A \ds M$ by $(2)$. Suppose that every submodule of $M$ generated by less than $n$ elements is a direct summand. Then the cyclic submodule $Ra_1$ is direct summand of $M$, i.e., $M = Ra_1 \oplus B$ for some $B \subseteq M$. By modular law, $A = Ra_1 \oplus (A \cap B).$  Then $A \cap B$ is generated by $p(a_2), p(a_3), \cdots, p(a_n)$  where $p : A \rightarrow A \cap B$ is the usual projection. By the hypothesis, $A \cap B \ds M$. Thus, $A \cap B \ds B$, i.e., $B = (A \cap B) \oplus C$ for some $C \subseteq B$. Therefore $$M = Ra_1 \oplus B = Ra_1 \oplus (A \cap B) \oplus C = A \oplus C.$$ So $A \ds M$.
\end{proof}

As pointed out, it is natural to ask whether the virtual version of Theorem \ref{strreg} is true, that is, if every cyclic submodule is isomorphic to a direct summand, does it imply that every finitely generated submodule is a direct summand? The following example shows that this is not the case.

\begin{example}\label{examplecvss}
Consider the polynomial ring $\mathbb{Z}[X]$ in one variable $X$ with integer coefficients. Since $\mathbb{Z}[X]$ is an integral domain, every principal ideal of $\mathbb{Z}[X]$ is isomorphic to $\mathbb{Z}[X]$.  On the other hand, it is well-known that $\Z[X]$ is Noetherian ring which not a PID, for example the ideal $I=(2,\, X)$ is finitely generated but not principal.   Then $I$ is not isomorphic to $\Z[X]$, and so $I$ is not isomorphic to a direct summand of $\Z[X]$.
\end{example}


 Theorem \ref{strreg} and Example \ref{examplecvss} motivate the following definitions.
 
\begin{defn} Let $M$ be a right $R$-module. 
\begin{enumerate}
  \item[(1)] $M$ is said to be {\it virtually regular} if every cyclic submodule of $M$ is isomorphic to a direct summand of $M$.
   \item[(2)] $M$ is  {\it completely virtually regular} if every  submodule of $M$ is virtually regular.
    \item[(3)] $M$ is {\it strongly virtually regular} if every finitely generated  submodule of $M$ is isomorphic to a direct summand of $M$.

\end{enumerate}

\end{defn}

The following examples distinguish the notions of virtually semisimple, virtually regular, strongly virtually regular and completely virtually regular modules.


\begin{example}
\begin{enumerate}
  \item Every virtually semisimple module is (strongly) virtually regular.
  \item Every domain is right and left virtually regular. Because for every $0\neq a \in R$, $aR\cong R$ and $Ra \cong R.$
  \item The $\Z$-module $M=\Z_2 \oplus \Z_4$ is virtually regular but it is not virtually semisimple: Cyclic submodules of $M$ are isomorphic to $0,\,\Z_2$ or $\Z_4$, and so $M$ is virtually regular. On the other hand, the submodule $\Z_2 \oplus 2\Z_4$ is not isomorphic to a direct summand of $M$, and so $M$ is not strongly virtually regular.
 

  \item The $\Z$-module $M = \Q \oplus \Z^{(\N)}$ is strongly virtually regular but it is not virtually semisimple: Let $C$ be a finitely generated  submodule of $M$. Then $C\cong \Z^n$ for some nonnegative integer $n$. Clearly $M$ has a direct summand isomorphic to $C$. Thus $M$ is strongly virtually regular. On the other hand, the submodule $N$ of  $\Q$ generated by $S=\{\frac{1}{2^k} \mid  k \in \Z ^+ \}$ is not finitely generated and it is not a free module. Now, $N$ can be embedded in $M$, and $M$ has no any direct summand isomorphic to $N$. Thus $M$ is not virtually semisimple.

\end{enumerate}
\end{example}

A ring $R$ is called \textit{right $pp$-ring} if every principal right ideal of $R$ is projective. $R$ is said to be \textit{right B\'{e}zout ring} if every finitely generated right ideal is principal. 

The following result address the question that when the ring is (strongly) virtually regular and completely virtually regular as right module over itself.

\begin{prop}\label{prop:Rvirtuallyregular} For a ring $R$, the following are hold.
\begin{enumerate}
\item[(1)] $R$ is right virtually regular if and only if $R$ is right pp-ring.
\item[(2)] $R$ is strongly virtually regular if and only if $R$ is  right  pp-ring and right B\'{e}zout ring.
\item[(3)] If $R$ is a domain, then $R$ is right completely virtually regular if and only if $R$ is a right PID.
\end{enumerate}
\end{prop}

\begin{proof}$(1)$  Suppose $R$ is right virtually regular. Let $I$ be a principal right ideal of $R$. Then $I \cong I'$ where $I' \subseteq^{\oplus} R$. Since $R_R$ is free and $I' \subseteq^{\oplus} R$, $I'$ is projective. Hence $I$ is projective. Thus $R$ is right pp-ring. For the converse, since every principal right ideal $I$ of $R$ is projective, $I$ is isomorphic to a direct summand of $R$. Thus $R$ is right virtually regular.

$(2)$ Suppose $R$ is right strongly virtually regular.  Let $I$ be a finitely generated right ideal of $R$. Hence $I \cong I'$ where $I' \subseteq^{\oplus} R$. Since $I'$ is principal and projective, $I$ is also principal and projective. Thus $R$ is right B\'{e}zout ring and right pp-ring.

$(3)$ Suppose $R$ is right completely virtually regular. Let $I$ be a nonzero right ideal and $0 \neq a \in I$. Then $aR$ is isomorphic to a direct summand of $I$, because $I$ is virtually regular. Since $R$ is a domain, $I$ is indecomposable. So that $aR\cong I$ is principal. Thus $R$ is a right PID, and this proves the necessity. The sufficiency is clear.
\end{proof}

Every domain is right $pp$-ring. Hence every domain is virtually regular, and a domain $R$ is strongly virtually regular if and only if $R$ is a B\'{e}zout domain.

\begin{prop} The following statements are equivalent for a ring $R:$
\begin{enumerate}
  \item Every right $R$-module is virtually regular.
  \item Every injective right $R$-module is virtually regular.
  \item $R$ is semisimple artinian.
\end{enumerate}
\end{prop}

\begin{proof}$(1)\Rightarrow (2)$ and $(3) \Rightarrow (1)$ are clear. Let us prove that $(2)\Rightarrow (3)$ by showing that each cylic right $R$-module is injective. Let $C$ be a cyclic right $R$-module and $E=E(C)$ be the injective hull of $C$. By $(2)$, there is a decomposition $E=K \oplus L$ such that $C\cong K.$ Since $K \ds E$, it is injective. Therefore $C$ is injective. Then $R$ is semisimple artinian by Osofsky's theorem in \cite{O}. This proves $(3).$
\end{proof}



\begin{lemma} If every cyclic right $R$-module is virtually regular, then $R$ is a right $V$-ring. 
\end{lemma}

\begin{proof} Suppose the contrary and let $S$ be a simple noninjective right $R$-module. Let $A$ be a cyclic submodule of $E(S)$ which properly contains $S$. Then $A$ has no direct summand isomorphic to $S$, and so $A$ is not virtually regular. Thus the proof follows.
\end{proof}
Now, we shall give an example of a nonsemisimple ring whose cyclic right modules are virtually regular.
A domain $R$ is said to be \textit{right PCI domain} if every proper cyclic right $R$-module is (semisimple) injective. Over a right PCI domain, every singular right $R$-module is semisimple injective. Moreover, over a right PCI domain, every right $R$-module $M$ has a decomposition as $M=S\oplus N$, where $S$ is semisimple injective and $N$ is nonsingular (see, \cite{Faith}, \cite{Huynh}, \cite{Jain}).

We obtain the following characterization of virtually regular modules over right PCI domains:

\begin{prop} Let $R$ be a right PCI domain. A right $R$-module $M$ is virtually regular if and only if $M$ is semisimple or has a direct summand isomorphic to $R$.
\end{prop}

\begin{proof}  Let $M$ be a virtually regular right $R$-module. Suppose $M$ is not semisimple. Then $M=S \oplus N$, where $N$ is nonzero, $S=Z(M)$ is semisimple, and $Z(N)=0$. Let $C$ be a nonsingular cyclic submodule of $M$. Then $C \cong R$, and $C \cong B \ds M$, because $M$ is virtually regular. Hence $M$ has a direct summand isomorphic to $R$. This proves the necessity.

For the proof of the sufficiency, let $C$ be a cyclic submodule of $M$. If $C$ is proper cyclic i.e. $C\cong \frac{R}{I}$ for some $I \neq 0$, then $C$ is injective and so $C \ds M$. If $C \cong R$, then $C \cong B \ds M$ by the assumption that $M$ has a direct summand isomorphic to $R$. Therefore $M$ is virtually regular.
\end{proof}

A cyclic right $R$-module over a PCI domain $R$ is either semisimple or it is isomorphic to $R$, both of which are virtually regular. Hence we have the following.

\begin{corollary} Over a right PCI domain $R$, every cyclic right $R$-module is virtually regular.
\end{corollary}

\begin{prop}\label{prop:Z(M)virtuallyregular} If $M$ is a virtually regular right $R$-module, then the singular submodule $Z(M)$ is virtually regular.
\end{prop}

\begin{proof} Let $C$ be a cyclic submodule of $Z(M)$. Then $C$ is a cyclic submodule of $M$ as well, and so $C\cong B \ds M$ by the assumption. Let $M=B\oplus D$ for some submodule $D$ of $M$. Since singular modules are closed under isomorphism, $B\subseteq Z(M)$. Then, by the modular law, $Z(M)=B\oplus D\cap Z(M).$ Thus, $Z(M)$ is virtually regular.
\end{proof}

Note that Proposition \ref{prop:Z(M)virtuallyregular} is also hold if one replace virtually regular by strongly virtually regular or completely virtually regular. The proof is similar, so it is omitted.

\begin{lemma}\label{lem:regularmodulesarecompletelyvirtuallyregular} Let $M$ be a regular right $R$-module. Then $M$ is completely virtually regular.
\end{lemma}

\begin{proof} Let $N$ be a submodule of $M$, and $C$ be a cyclic submodule of $N$. Since $M$ is regular, $M=C \oplus D$ for some $D \subseteq M$. Then, by the modular law, $N=C \oplus (D\cap N)$. Therefore $N$ is virtually regular. Hence $M$ is completely virtually regular.
\end{proof}

Every finitely generated virtually semisimple right module has finite uniform dimension. But this is not the case for  finitely generated (completely) virtually regular modules as it is shown in the following example.

\begin{example}(\cite[Definition 2.2 and Lemma 2.3]{trlifaj}) Let $K$ be a field, and $R$ the subalgebra of $K^{\omega}$ consisting of all eventually constant sequence in $K^{\omega}$.
In other words, $R$ is the subalgebra of $K^{\omega}$ consisting of all eventually
constant sequences in $K^\omega$. For each $i < \omega$, we let $e_i$ be the idempotent in $K^{\omega}$ whose $i$th component is 1
and all the other components are 0. Notice that $\{e_i \mid i < \omega \}$ is an infinite set of pairwise
orthogonal idempotents in $R$. The ring $R$ is a commutative von Neumann regular semiartinian ring of Loewy length
2, with $soc(R) = \sum_ {i< \omega} e_iR = K^{(\omega)}$ and $ \frac{R}{soc(R)} \cong K$. Thus $R$ is a regular ring and $R$ has no finite uniform dimension. Note that $R$ is completely virtually regular by Lemma \ref{lem:regularmodulesarecompletelyvirtuallyregular}.
\end{example}

A right $R$-module is said to be \textit{virtually simple} if every nonzero submodule is isomorphic to $M$. It is easy to see that,  a $\Z$-module is virtually simple if and only if it is isomorphic to $\Z$ or to $\Z_p$ for some prime $p$. Clearly, a domain $R$ is virtually simple if and only if it is a PID. More generally:

\begin{lemma} Let $M$ be a right $R$-module with $u.dim(M)=1.$ $M$ is virtually regular if and only if $M$ is virtually simple.
\end{lemma}

It is natural to ask, whether a virtually regular right module with finite uniform dimension is a direct sum of virtually simple modules. The following example shows that this is not the case.

\begin{example} The $\Z$-module $M=\Q \oplus \Z$ has uniform dimension two, and it is virtually regular. But, $\Q$ is not virtually simple.
\end{example}

\begin{prop}\label{prop:svrudim} Let $M$ be a strongly virtually regular module with finite uniform dimension. Then $M$ is finitely generated.
\end{prop}

\begin{proof} Suppose that $M$ has a finite uniform dimension, say $n$. Then there is an essential submodule $N=a_1R \oplus \cdots \oplus a_nR$ of $M$. By the hypothesis, $N$ is isomorphic to a direct summand of $M$, that is $M=K\oplus L$ for some submodules $K,\,L$ of $M$ with $N\cong K.$ Clearly $M$ and $K$ have the same uniform dimension, and so $L$ must be the zero module. Thus $M=K$, and hence  $M$ is finitely generated.
\end{proof}

The converse of Proposition \ref{prop:svrudim} is not true. For example every regular ring is strongly regular, and the uniform dimension of nonsemisimple regular rings is not finite.

\begin{prop} The following are hold.
\begin{enumerate}
\item[(1)] A Noetherian right module $M$ is virtually semisimple if and only if $M$ is strongly virtually regular.

  \item[(2)] Any Artinian strongly virtually regular right module $M$ is semisimple.

\end{enumerate}
\end{prop}

\begin{proof} $(1)$ This is clear, since every submodule of $M$ is finitely generated.

$(2)$ Since $M$ is Artinian, $\Soc(M)$ is finitely generated. Then, by the hypothesis, $M=K\oplus N$ for some $K \cong \Soc(M)$. Every nonzero Artinian module has nonzero socle, thus $\Soc(M)=\Soc(N)$ implies that $N=0.$ Therefore, $M=\Soc(M)$ is semisimple.
\end{proof}

\begin{defn}
Let $M$ be a right $R$-module.
\begin{enumerate}
    \item  $M$ is said to be \textit{retractable} if for any nonzero submodule $N$ of $M$, $Hom_R(M, N) \neq 0$; and $M$ is called \textit{epi-retractable} if for every submodule $N$ of $M$, there exists an epimorphism $f : M \rightarrow N$. We call $M$ \textit{cyclic epi-retractable} if for every cyclic submodule $A$ of $M$, there exists an epimorphism $f : M \rightarrow A$.
\item Let $P$ and  $M$ be $R$-modules. $P$ is said to be \textit{$M$-projective}, if for every module $N$ and every epimorphism $g:M \to N$ and a homomorphism $f: P \to N$, there is a homomorphism $h: P\to M$ such that $f=gh.$ If $P$ is $P$-projective, then $P$ is called \textit{quasi-projective}. We call $P$ \textit{cyclic $M$-projective} if the said property hold for each cyclic module $N$, and we call $P$ \textit{weak quasi-projective} if $P$ is cyclic $P$-projective.

Every nonzero module with $rad(M)=M$ is weak quasi-projective because $M$ has no (nonzero) cyclic factor modules. In particular, $\Q$ is weak quasi-projective as a $\Z$-module but $\Q$ is not quasi-projective (see, \cite[Theorem 1]{ranga}).
\end{enumerate}
\end{defn}

Every ring is left and right cyclic epi-retractable. Every epi-retractable right module is cyclic epi-retractable. A ring $R$ is right epi-retractable if and only if $R$ is a right principal ideal ring. Hence any ring which is not a right principal ideal ring is an example of a cyclic epi-retractable ring that is not right epi-retractable. Every virtually regular right module is cyclic epi-retractable. But the converse is not true in general. For example, it is easy to see that $\Z_{p^2}$ is cyclic epi-retractable but not virtually regular.


Properties of virtually semisimple modules in \cite[Proposition 2.1]{vss} are adapted to (completely) virtually regular modules as follows.
\begin{prop}\label{prop:mainproperties}
The following statements hold for a right $R$-module $M$.
\begin{enumerate}
  \item Suppose $M$ is (strongly) virtually regular  module. If $M = M_1 \oplus M_2$ is a decomposition for $M$ such that $Hom_{R}(M_1, M_2) = 0$, then $M_2$ is a (strongly) virtually regular $R$-module.

  \item Let $I$ be an ideal of $R$ with $MI=0$. Then $M$ is a (completely) virtually regular $\frac{R}{I}$-module if and only if $M$ is (completely) virtually regular as an $R$-module.

  \item  $M$ is virtually regular weak quasi-projective if and only if it is a cyclic epi-retractable  and all of its cyclic submodules are $M$-projective.

  \item Suppose $M$ is virtually regular and $A \subseteq M$. If $A$ contains any cyclic submodule $K$ of $M$ with $K$ is embedded in $A$, then $A$ is virtually regular. 
\end{enumerate}
\end{prop}
\begin{proof} (1) and (2) are similar to proof of (i) and $(ii)$ in \cite[Proposition 2.1.]{vss}.

      (3) Suppose $M$ is virtually regular weak quasi-projective. Let $N$ be a cyclic submodule of $M$. Then $N \cong K \ds M$ for some $K$ of $M$. Thus there is an epimorphism $M \to N$. So $M$ is cyclic epi-retractable. On the other hand, the class of weak quasi-projective modules is closed under direct summands and isomorphisms. Therefore $K$, and hence $N$ is $M$-projective. This proves the necessity.
      Conversely, suppose that $M$ is cyclic epi-retractable and every cyclic submodule of $M$ is $M$-projective. Let $N$ be a cyclic submodule of $M$. Since $M$ is cyclic epi-retractable, there is an epimorphism $f:M \to N$. Also $N$ is $M$-projective by the hypothesis. Therefore there is a homomorphism $g: N \to M$ such that $fg=1_N$. Hence $N$ is isomorphic to a direct summand of $M$, so $M$ is virtually regular.

  Now, let us show that $M$ is weak quasi-projective. Let $B$ be a cyclic submodule of $M$, and $g:M\to B$ be an epimorphism. Since $B$ is $M$-projective, the map $g$ splits i.e. there is $h:B\to M$ such that $gh=1_B.$ Now if $f:M\to B$  is any homomorphism, then the map $hf: M \to M$ satisfies $g(hf)=f$. Thus, $M$ is weak quasi-projective.


 (4) Let $A$ be a submodule of $M$ and $K$ be a cyclic submodule of $A$. Since $M$ is virtually regular, $K \cong L$ where $L \subseteq^{\oplus} M$. Then $L$ is also cyclic, and by the hypothesis $L \subseteq A$. Thus, by the modular law, $L \subseteq^{\oplus} A$, and hence $A$ is virtually regular.
\end{proof}

An $R$-module $M$ is said to be \textit{Rickart module} if for each $\phi \in End_R(M)$, then $\Ker \phi \ds M$. Weaking this notion, we call $M$\textit{ weak Rickart module} if for each $\phi \in End_R(M)$ with $\phi (M)$ cyclic, then $\Ker \phi \ds M$.

The following is the analogue of \cite[Proposition 2.2.]{vss} for virtually regular modules.

\begin{prop}
A weak quasi-projective $R$-module $M$ is virtually regular if and only if it is a cyclic epi-retractable weak Rickart $R$-module.
\end{prop}
\begin{proof}
Assuming $M$ is virtually regular gives $M$ is cyclic epi-retractable by Proposition \ref{prop:mainproperties}. Now, let $f \in End(M)$ with $f(M)$ is cyclic. Then by assumption, $f(M)$ is isomorphic to a direct summand of $M$. Hence it is $M$-projective. It follows that $Kerf$ is a direct summand of $M$, proving that $M$ is Rickart. Suppose, for the converse, $M$ is a cyclic epi-retractable and weak Rickart $R$-module. Let $A$ be a cyclic submodule of $M$. Then there is an epimorhism $f : M \rightarrow A$. Since $M$ is weak quasi-projective, there is a homomorphism $g : A \rightarrow M$ such that $f g = id_{A}$. Therefore $M = Kerf \oplus Img$, and hence $$A \cong \frac{M}{Kerf} \cong Img \subseteq^{\oplus} M.$$ Thus $M$ is virtually regular.
\end{proof}

\section{Modules over commutative rings}
In this section, we study (completely) virtually regular modules over commutative rings.  The rings considered in this section are commutative, unless otherwise stated.

We begin with the following.

\begin{prop}\label{lem:cyclicvirtuallyregular} The following statements are equivalent for a nonzero indecomposable module $A$ over a commutative ring $R$.
\begin{enumerate}
  \item $A$ is  virtually regular.
  \item $A \cong \frac{R}{P}$ for some prime ideal $P$.
  \item Every nonzero cyclic submodule of $A$ is isomorphic to $A$.
\end{enumerate}
\end{prop}

\begin{proof}$(1) \Rightarrow (2)$ Let $C$ be a nonzero cyclic submodule of $A$. Then $C\cong B \ds A$ for some $B \subseteq A.$ Since $A$ is indecomposable, $B=A$ is cyclic. Suppose $A\cong \frac{R}{P}$ for some ideal $P$ of $R$. Lest us show that $P$ is a prime ideal of $R$. For this purpose, assume that $ab\in P$ for some $a,b \in R$. If $a$ is not contained in $P$, then $R(a+P) \cong \frac{R}{P}$ by $(1)$. Thus $ann_R(R(a+P))=P$. As $ab\in P$ and $R$ is commutative, we have $b \in ann_R(a+P)=P.$ Therefore $P$ is a prime ideal of $R$.

$(2) \Rightarrow (3)$ Let $(a+P)$ be a nonzero element of $\frac{R}{P}$. Then $a$ is not in $P$. Suppose that, $R(a+P)$ is not isomorphic to $\frac{R}{P}.$  Then $R(a+P)\cong \frac{R}{Q}$ for some $P\subset Q \subseteq R$. Then $Qa \subseteq P$, and so primality of $P$ implies that $Q \subseteq P$,  a contradiction. Hence $R(a+P)\cong \frac{R}{P}$, and so $(3)$ holds.

$(3) \Rightarrow (1)$ is clear.
\end{proof}

\begin{prop}
Let $R$ be a commutative ring and $B$ be an indecomposable $R$-module. The following statements are equivalent.
\begin{enumerate}
\item $B$ is completely virtually regular.
\item $B \cong \frac{R}{I}$ is a principal ideal domain.
\item $B$ is virtually semisimple.
\end{enumerate}
\end{prop}

\begin{proof} $(1) \iff (2)$ Completely virtually regular indecomposable module is cyclic. Then the proof follows by Proposition \ref{prop:mainproperties}(2) and Proposition \ref{prop:Rvirtuallyregular}(3). 

$(2) \iff (3)$ An indecomposable virtually semisimple module is cyclic. Hence, the proof follows by \cite[Theorem 2.3]{svss}.

\end{proof}

Although the structure of virtually regular torsion modules over commutative domains is not completely known, we have the following for torsion-free modules.

\begin{prop}\label{torsion-freecvssoverdomain}
Let $R$ be a domain and $M$ a torsion-free $R$-module. Then the following statements are equivalent.
\begin{enumerate}
  \item $M_R$ is virtually regular.
  \item $M$ has a direct summand isomorphic to $R$.
\end{enumerate}
\end{prop}
\begin{proof}

$(1) \Rightarrow  (2)$ Suppose $M_R$ is virtually regular. Let $aR$ be a nonzero cyclic submodule of $M_R$. Since $R$ is a domain and $M$ is torsion-free, $aR \cong R$ is projective. By $(1)$, $M$ is virtually regular, and so $aR\cong L \ds M$ for some submodule $L$ of $M$. Thus $L\cong R$, and so $(2)$ follows.
$(2) \Rightarrow  (1)$ Let $A$  be a cyclic submodule of $M$. Then $A \cong R$, because $R$ is a domain and $M$ is torsion-free. By $(2)$, $M$ has a direct summand isomorphic to $R$. Hence $A$ is isomorphic to a direct summand of $M$. Therefore $M$ is virtually regular.
\end{proof}



For completeness,  we recall the following definitions from \cite{FuchsAndSalce:ModulesOverNonNoetherianDomains}.

\begin{defn}
\begin{enumerate}
\item An $R$-module $M$ is called \textit{uniserial} if its submodules are totally ordered by inclusion; equivalently, for all $a$, $b$ in $M$, $$aR \subseteq bR\text{ or } bR \subseteq aR.$$
\item A  domain $R$ is said to be a \textit{valuation domain} if $R$ is uniserial.  $R$ is a \textit{Pr\"ufer domain} if all its localizations at maximal ideals are valuation domains.

 
 \item A valuation ring is called \textit{maximal} if a system of congruences \begin{equation}\label{eee}
 x \equiv r_k \ \ (mod \ \  L_k) \ \ \ (k \in K) \end{equation} has a solution in $R$ whenever each of its finite sub-system is solvable in $R$,  where $L_k$'s are ideals of $R$, $r_k \in R$ and $K$ is an index set . Also, $R$ is  said to be \textit{ almost maximal} if the above congruences \ref{eee} have a solution whenever $\cap_{k \in K}L_k \neq 0$.
 
 
\end{enumerate}
\end{defn}

A right $R$ module $M$ is \textit{pure-injective} if for every right $R$-module $B$ and pure submodule $A$ of $B$, every homomorphism $f:A \to M$ can be extended to a homomorphism $g: B \to M$. We recall the following result.

\begin{thm}\cite[Theorem 4.6]{FuchsAndSalce:ModulesOverNonNoetherianDomains}\label{almostprüfer}
Uniserial (torsion) modules over an (almost) maximal Pr\"ufer domain are pure-injective.
\end{thm}

More generally, one has the following.

\begin{corollary}\label{cor:propercyclicpureinjective} Over an almost maximal Pr\"ufer domain, every proper cyclic module is pure-injective.
  \end{corollary}

\begin{proof}Let $R$ be almost maximal Pr\"ufer and $A\cong \frac{R}{I}$, $I \neq 0$, be a proper cyclic $R$-module. Then $A\cong A_{P_1}\oplus \cdots \oplus A_{P_n}$ for some maximal ideals $P_1,\cdots ,P_n$ by \cite[Theorem 3.7, Chap.IV]{FuchsAndSalce:ModulesOverNonNoetherianDomains} and \cite[Theorem 3.9(a), Chap.IV]{FuchsAndSalce:ModulesOverNonNoetherianDomains}. On the other hand, each $R_{P_i}$ is almost maximal valuation domain by \cite[Theorem 3.9(b), Chap.IV]{FuchsAndSalce:ModulesOverNonNoetherianDomains}. Since $A$ is a proper cyclic module, $A_{P_i}$ is a proper cyclic $R_{P_i}$-module for each $i=1,\cdots,n.$ Thus each $A_{P_i}$ is a uniserial torsion cyclic $R$-module. Then $A_{P_i}$ is pure-injective $R$-module by \cite[Theorem 4.6, Chap.XIII]{FuchsAndSalce:ModulesOverNonNoetherianDomains}. Thus $A$ is pure-injective $R$-module as a finite direct sum of pure-injective $R$-modules. This completes the proof.
\end{proof}

\begin{thm}\label{prop:torsiontorsionfreepart} Let $R$ be a domain and $M$ be an $R$-module. Consider the following statements.
\begin{enumerate}

\item[(1)] $M$ is virtually regular.

\item[(2)] $T(M)$ and $\frac{M}{T(M)}$ are virtually regular.

\end{enumerate}

Then, $(1)\Rightarrow (2).$ Moreover, if $R$ is an almost maximal Pr\"ufer domain, then $(2)\Rightarrow (1)$.
\end{thm}

\begin{proof} $(1)\Rightarrow (2)$ Let $A$ be a cyclic submodule of $T(M)$. Then $A$ is a cyclic submodule of $M$, and so $A\cong B \ds M$ by $(a).$ Since $A$ is torsion, $B$ is torsion as well. That is, $B\subseteq T(M)$. As $M=B\oplus C$, by modular law, we get $T(M)=B \oplus C\cap T(M).$
Therefore, $T(M)$ is virtually regular.\\ Now, let us show that $\frac{M}{T(M)}$ is virtually regular. Since $M$ is virtually regular, $M$ has a direct summand isomorphic to $R$. So there is an epimorphism $g: M \to R$. As $R$ is torsionfree, $T(M)\subseteq \Ker(g).$ So that $g$ induces an epimorphism $g': \frac{M}{T(M)}\to R$. Then $g'$ splits by projectivity of $R$. Hence, $\frac{M}{T(M)}$ has a direct summand isomorphic to $R$, and so $\frac{M}{T(M)}$ is virtually regular by proposition \ref{torsion-freecvssoverdomain}.

$(2)\Rightarrow (1)$ Suppose that $R$ is almost maximal Pr\"ufer. Let $A$ be a cyclic submodule of $M$. Then $A\cong R$ or $A\cong \frac{R}{I}$ for some proper ideal $I$ of $R$. Note that in the latter case $A \subseteq T(M).$ First, assume that $A \subseteq T(M).$ Then $A\cong B\ds T(M)$, by $(2)$. Thus $B$ is a pure submodule of $T(M)$, and as $R$ is Pr\"ufer $T(M)$ is a pure submodule of $M$ by \cite[Proposition 8.12]{FuchsAndSalce:ModulesOverNonNoetherianDomains}. Therefore, $B$ is a pure submodule of $M$. On the other hand, as $B$ is isomorphic to  a proper cyclic module, it is  pure-injective by Corollary \ref{cor:propercyclicpureinjective}. Being pure-injective and a pure submodule of $M$, $B$ must be a direct summand of $M$. Thus, we obtain that, $A\cong B \ds M$. \\
Now, assume that $A\cong R.$ Since $\frac{M}{T(M)}$ is torsionfree and virtually regular, it has a direct summand isomorphic to $R$ by proposition \ref{torsion-freecvssoverdomain}. Then there is an epimorphism $f:\frac{M}{T(M)}\to R.$ This leads to an epimorphism $f\pi: M \to R$, where $\pi: M \to \frac{M}{T(M)}$ is the canonical epimorphism. As $R$ is projective, $f\pi$ splits, and so $M$ has a direct summand isomorphic to $R$. Hence $A$ is isomorphic to a direct summand of $M$, and so $M$ is virtually regular. This proves $(1).$
\end{proof}


Over the ring of integers, every proper cyclic is pure-injective (see, \cite[Theorem 30.4]{fuchs}). Hence, we have the following consequence of Theorem \ref{prop:torsiontorsionfreepart}.

\begin{corollary}\label{cor:abeltorsiontorsionfree} An abelian group $M$ is virtually regular if and only if $T(M)$ and $\frac{M}{T(M)}$ are virtually regular.
\end{corollary}



Let $R$ be a Dedekind domain and $I$ be a nonzero proper  ideal of $R$. Then $I=P_1^{n_1}\cdots P_k ^{n_k},$ where $P_1,\cdots P_k$ are distinct maximal ideals and $n_1,\cdots, n_k$ are positive integers (see \cite[Corollary 7.5]{passman}).

\begin{prop}\label{prop:torsionoverdedekinddomain} Let $R$ be a domain and $M$ be an $R$-module. The following are hold.
\begin{enumerate}

\item[(1)] If $M$ is completely virtually regular, then $T(M)$ is completely virtually regular.
\item[(2)] If $R$ is a Dedekind domain and $M$  a torsion $R$-module, then $M$ is  completely virtually regular if and only if $M$ is semisimple.

\end{enumerate}
\end{prop}

\begin{proof} $(1)$ For domains, the singular and the torsion submodule coincide. Then proof follows by Proposition \ref{prop:Z(M)virtuallyregular}.

$(2)$ Sufficiency is clear. To prove the necessity, let $C \cong \frac{R}{I}$ be a nonzero cyclic submodule of $M$. Then $I \neq 0$, because $M$ is a torsion module. Thus $I = P_1^{n_1}\cdots P_k ^{n_k}$ for distinct maximal ideals $P_1,\cdots P_k$. Then $$\frac{R}{I} \cong \frac{R}{P_1 ^{n_1}}\oplus \cdots \oplus \frac{R}{{P_k}^{n_k}}.$$
By the assumption, each $\frac{R}{{P_i}^{n_i}}$ is virtually regular. Since $\frac{R}{{P_i}^{n_i}}$ is indecomposable, $P_i^{n_i}$ is a prime ideal by Proposition \ref{lem:cyclicvirtuallyregular}, and so $n_i =1$ for each $i=1,\cdots ,k.$ Therefore, by maximality of $P_i$,  $\frac{R}{I} \cong \frac{R}{P_1}\oplus \cdots \oplus \frac{R}{P_k}$ is semisimple. Hence each cyclic submodule of $M$ is semisimple, and so $M=\bigoplus _{m\in M} mR$ is semisimple.
\end{proof}

Dedekind domains are almost maximal Prüfer domains by \cite[Theorem 3.9]{FuchsAndSalce:ModulesOverNonNoetherianDomains}. This fact together with Proposition \ref{torsion-freecvssoverdomain},  Theorem \ref{prop:torsiontorsionfreepart} and Proposition \ref{prop:torsionoverdedekinddomain} give the following.

\begin{corollary}\label{cor:cvroverDD} Let $R$ be a Dedekind domain and $M$ be a right $R$-module. The following statements are equivalent.
\begin{enumerate}
\item[(1)] $M$ is completely virtually regular.
\item[(2)] $(i)$   $T(M)$ is semisimple and 

\noindent $(ii)$ every submodule $N\subseteq M$ with $N \neq T(N)$ has a direct summand isomorphic to $R$.
\end{enumerate}
\end{corollary}







\section{Valuation Domains}

 In this section, we shall characterize finitely presented  virtually regular, strongly virtually regular  and completely virtually regular modules over valuation domains.
Recall that,  a right $R$-module is \textit{finitely presented} if $M \cong \frac{F}{K}$, where $F$ is a finitely generated free right $R$-module and $K$ a finitely generated submodule of $F$. The following characterization of finitely presented modules over valuation domains is due to Warfield. 

\begin{lemma}\cite[Theorem 1]{war} \label{lem:FPovervaluationdomains} Let $R$ be a valuation domain. An $R$-module $M$ is finitely presented if and only if $$M\cong R^m \oplus \frac{R}{Ra_1}\oplus \cdots \oplus \frac{R}{Ra_n}$$where $Ra_1 \geq Ra_2 \geq \cdots \geq Ra_n.$
\end{lemma}

 We recall the following lemma, which will be used in the sequel.

\begin{lemma}\label{lemma:primeidealprincipal}(\cite[Proposition 5.2.(f)]{facchini}) Let $R$ be a valuation domain with the unique maximal ideal $P$. If $A$ is a prime ideal of $R$ which is principal, then either $A=P$ or $A=0.$
\end{lemma}

A module $M$ over a commutative domain is said to be \textit{mixed module}, if $T(M)$ and $\frac{M}{T(M)}$ are both nonzero.  The following result is crucial for determining the structure of virtually regular modules over valuation domains. 

\begin{prop}\label{prop:valuationadmitstorsionvr} Let $R$ be a valuation domain and $P$ be its unique maximal ideal. The following statements are equivalent.

\begin{enumerate}
  \item There is a finitely presented mixed virtually regular $R$-module.
  \item There is a finitely presented nonzero torsion virtually regular $R$-module.
  \item There is a finitely presented nonzero cyclic torsion virtually regular $R$-module.
  \item $P=Rp$ is cyclic.
\end{enumerate}

\end{prop}

\begin{proof}$(1) \Rightarrow (2)$ By Theorem \ref{prop:torsiontorsionfreepart}.

$(2) \Rightarrow (3)$ is clear.

$(3) \Rightarrow (4)$ Suppose $M=\frac{R}{Ra}$ is virtually regular for some nonunit element $a$ of $R$. Since $R$ is a valuation domain, every cyclic $R$-module is uniserial. Thus $\frac{R}{Ra}$ indecomposable, and so $Ra$ is a prime ideal of $R$ by Proposition \ref{lem:cyclicvirtuallyregular}. By lemma \ref{lemma:primeidealprincipal}, a nonzero cyclic prime ideal is equal to $P$. Therefore $P=Ra$. Setting $P=Rp$, gives $(4).$

$(4) \Rightarrow (1)$ Consider the $R$-module $M=R\oplus \frac{R}{Rp}$, where $P=Rp$. Then $M$ is a mixed finitely presented $R$-module. Note that $R$ and $\frac{R}{Rp}$ are indecomposable modules with local endomorphism rings. Suppose $A\cong \frac{R}{I}$ is a nonzero cyclic submodule of $M$. If $I=0,$ then $A\cong R$. If $I\neq 0$, then $A \subseteq T(M)=0\oplus\frac{R}{Rp}.$ Then $A\cong \frac{R}{Rp}.$ In both cases, $A$ is isomorphic to a direct summand of $M$. Hence $M$ is virtually regular.

\end{proof}

\begin{lemma}\label{lem:ppowersarevr} Let $R$ be a valuation domain and $P=Rp$ be its unique maximal ideal. Then $$M=R^{(I)} \oplus (\frac{R}{Rp})^{(I_1)} \oplus (\frac{R}{Rp^2})^{(I_2)} \oplus \cdots \oplus (\frac{R}{Rp^k})^{(I_k)}\oplus \cdots=R^{(I)} \oplus (\bigoplus_{k=1}^{\infty}(\frac{R}{Rp^k})^{(I_k)})$$is virtually regular for each nonempty index sets $\,I_1,\,I_2,\cdots$.
\end{lemma}

\begin{proof}Let $A$ be a cyclic submodule of $M$. Then $A\cong R$ or $A\cong \frac{R}{L} \subseteq T(M).$ In the former case, clearly $A$ is isomorphic to a direct summand of $M$. In the latter case, let $t$ be the  smallest positive integer such that $p^t A=0$. Then $L=ann_R(A)=Rp^t$, and so $A\cong \frac{R}{Rp^t}$, and so $A$ is isomorphic to a direct summand of $M$. Therefore, $M$ is virtually regular.
\end{proof}

In the following theorem, which is one of the main result of the article, we characterize finitely presented  virtually regular modules over valuation domains. As we see, the unique maximal ideal $P$ of the valuation domain $R$ plays a crucial role in this characterization. 

\begin{thm}\label{thm:fpvirtuallyregularovervd}  Let $R$ be a valuation domain and $P$ be its unique maximal ideal. Then

\begin{enumerate}
\item[(1)] Every free $R$-module is virtually regular.
 \item[(2)] Finitely presented virtually regular modules are free if and only if P is not principal.
 \item[(3)]   $P=Rp$ is principal if and only if  finitely presented  virtually regular modules are of the form $$R^n \oplus (\frac{R}{Rp})^{n_1} \oplus (\frac{R}{Rp^2})^{n_2} \oplus \cdots \oplus (\frac{R}{Rp^k})^{n_k}$$for nonnegative integers $n,\,k,\,n_1,\,n_2,\cdots ,n_k.$
\end{enumerate}

\end{thm}

\begin{proof} $(1)$ By Proposition \ref{torsion-freecvssoverdomain}, every free $R$-module is virtually regular. 

$(2)$ Follows by $(1)$ and  Proposition \ref{prop:valuationadmitstorsionvr}.

$(3)$ Sufficiency follows from Proposition \ref{prop:valuationadmitstorsionvr}. For the necessity, assume that $P=Rp$ is principal and $$M\cong R^n \oplus (\frac{R}{Ra_1})^{n_1} \oplus \cdots \oplus (\frac{R}{Ra_k})^{n_k}$$ is a finitely presented virtually regular $R$-module, where $R > Ra_1 > Ra_2 > \cdots > Ra_k$ and $n_1,\cdots , n_k$ are positive integers.  Note that, by Proposition \ref{prop:torsiontorsionfreepart}, $$T(M)= (\frac{R}{Ra_1})^{n_1} \oplus \cdots \oplus (\frac{R}{Ra_k})^{n_k}$$ is virtually regular. We shall prove that, $Ra_i=Rp ^i$ for each $i=1,\cdots ,k$. First we claim that, $Ra_1$ is a prime ideal of $R$. For this, suppose $ab\in Ra_1$ and $a \notin Ra_1$. Then $A=R(a+Ra_1)$ is a cyclic submodule of $\frac{R}{Ra_1}.$ Identify $A$ with $i_1(A)$, where $$i_1: \frac{R}{Ra_1} \to (\frac{R}{Ra_1})^{n_1} \oplus \cdots \oplus (\frac{R}{Ra_k})^{n_k}=T(M)$$ is the canonical embedding.  Then $i_1(A)$ is a cyclic submodule of $T(M)$, and so $i_1(A)\cong B \ds T(M)$ because $T(M)$ is virtually regular. Clearly $B$ is cyclic, and as $R$ is a valuation domain it is indecomposable. Then by the Krull-Schmidt Theorem \cite[Theorem 9.8]{FuchsAndSalce:ModulesOverNonNoetherianDomains}, $B\cong \frac{R}{Ra_i}$ for some $i=1,\cdots, k.$ \\Since $A\cong i_1(A)\cong B$, we have $(Ra_1)B=0$. So that $Ra_1 \subseteq Ra_i=ann_R(B)$. Thus $Ra_1=Ra_i.$ Now $ann_R(A)=ann_R(i_1(A))=ann_R(B)=Ra_1$  and $bA=bR(a+Ra_1)=0$ together imply that, $b\in ann_R(A)=Ra_1.$ Thus $Ra_1$ is a nonzero prime ideal of $R$. Therefore $Ra_1=Rp$ by Lemma \ref{lemma:primeidealprincipal}. Now we proceed by induction on $k$. Suppose the claim is true for $k-1$, and $$T(M)=(\frac{R}{Rp})^{n_1} \oplus (\frac{R}{Rp^2})^{n_2}\oplus \cdots \oplus (\frac{R}{Rp^{k-1}})^{n_{k-1}}\oplus (\frac{R}{Ra_k})^{n_k}.$$ We need to prove that $Ra_k =Rp^k$. By Lemma \ref{lem:ppowersarevr}, the submodule $(\frac{R}{Rp})^{n_1} \oplus (\frac{R}{Rp^2})^{n_2}\oplus \cdots \oplus (\frac{R}{Rp^{k-1}})^{n_{k-1}}$ of $T(M)$ is virtually regular. Let $C$ be a nonzero cyclic submodule of $\frac{R}{Ra_k}$. If $i_k: \frac{R}{Ra_k} \to T(M)$ is the canonical embedding, then $C\cong i_k(C)$ is a cyclic submodule of $T(M).$ Thus $i_k(C)\cong D \ds T(M)$, because $T(M)$ is virtually regular. Now, $D$ is cyclic and indecomposable direct summand of $T(M)$, and so $D\cong \frac{R}{Ra_k}$ or $D \cong \frac{R}{Rp^i}$ for some $i=1,\cdots, k-1$. Note that $C\cong i_k(C) \cong D$. If $\Soc (\frac{R}{Ra_k})=0$, then we must have $C\cong D\cong \frac{R}{Ra_k}$, because $\Soc (\frac{R}{Rp^i })\neq 0$ for each $i=1,\cdots, k-1$. Thus, in this case, $C\cong {\frac{R}{Ra_k}}$ for each nonzero cyclic submodule $C$ of $\frac{R}{Ra_k}$, as well. Then $Ra_k$ is a prime ideal of $R$ by Proposition \ref{lem:cyclicvirtuallyregular}. So that $Ra_k=Rp$ by Lemma \ref{lemma:primeidealprincipal}, a contradiction. Hence, we must have $\Soc(\frac{R}{Ra_k})=\frac{S}{Ra_k} \neq 0$. Note that, since $J(R)=Rp$ and $\frac{S}{Ra_k}$ is simple, $Rp(\frac{S}{Ra_k})=0.$ Thus $S\subseteq ann_R(\frac{Rp}{Ra_k})$ and $S\frac{R}{Ra_k}=\frac{S}{Ra_k}\neq 0$. Therefore, $\frac{R}{Ra_k}$ and $\frac{Rp}{Ra_k}$ are not isomorphic.  This implies that $i_k(\frac{Rp}{Ra_k})\cong \frac{R}{Rp^{k-1}}$, by the virtually regular assumption. So, the composition length of $\frac{Rp}{Ra_k}$   is $k-1.$  From the short exact sequence $$0 \to \frac{Rp}{Ra_k} \to \frac{R}{Ra_k}\to \frac{R}{Rp} \to 0,$$ we obtain that $p^kR=p^{k-1}(Rp)\subseteq Ra_k$, and the composition length of $\frac{R}{Ra_k}$ is $p^k$. Thus we must have $\frac{R}{Ra_k}\cong \frac{R}{Rp^k}$. This completes the proof.
\end{proof}

\begin{remark}
In \cite{svss}, the authors ask whether torsion-free virtually semisimple modules are free over a DVR. We do not know the answer to this question but the answer of the corresponding question for virtually regular and strongly virtually regular modules is no. For example, let $R$ be a DVR with quotient ring $K$. Consider the torsion-free modules $M=K \oplus R$ and $ N=K\oplus R^{(I)}$ for an infinite index set $I$. Then $M$ and $N$ are not free $R$-modules. Clearly, $M$ is virtually regular by Proposition \ref{torsion-freecvssoverdomain}. Let $B$ be  a finitely generated submodule of $N$. Then $B\cong R^n$ for some positive integer $n$, and so $B$ is isomorphic to a direct summand of $N$. Therefore $N$ is strongly virtually  regular.
\end{remark}

Theorem \ref{thm:fpvirtuallyregularovervd} in hand, we could obtain the structure of finitely presented strongly virtually regular and completely virtually regular  modules.

\begin{prop}\label{prop:svrovervd} Let $R$ be a valuation domain with unique maximal ideal $P$, and let $M$ be a finitely presented $R$-module. Then

\begin{enumerate}
\item[(1)] If $P$ is not principal, then $M$ is strongly virtually regular if and only if $M\cong R^n$.

\item[(2)] If $P=Rp$ is principal, then $M$ is strongly virtually regular if and only if $M\cong R^n \oplus (\frac{R}{Rp})^m$, where $n,\,m$ are nonnegative integers.
\end{enumerate}

\end{prop}

\begin{proof} $(1)$ Strongly virtually regular modules are virtually regular. Then  necessity follows by Theorem \ref{thm:fpvirtuallyregularovervd}(2). For the sufficiency suppose $M \cong R^n$, and let $N$ be a finitely generated submodule of $M$. Since valuation domains are semihereditary, $N$ is projective. As $R$ is a local ring, $N\cong R^m$ for some $m \leq n$. Hence $N$ is isomorphic to a submodule of $M$, and so $M$ is strongly virtually regular.

$(2)$ Suppose $M$ is strongly virtually regular, then by Theorem \ref{thm:fpvirtuallyregularovervd}(3), $$M \cong R^n \oplus (\frac{R}{Rp})^{n_1} \oplus (\frac{R}{Rp^2})^{n_2} \oplus \cdots \oplus (\frac{R}{Rp^k})^{n_k}$$for nonnegative integers $n,\,k,\,n_1,\,n_2,\cdots ,n_k.$   Then $\Soc (M)\cong (\frac{R}{Rp})^{n_1 +\cdots +n_k}.$  Assume that $k \geq 2.$ If $N\ds M$ and $N$ is semisimple, then $N$ is isomorphic to  a submodule of $(\frac{R}{Rp})^{n_1}$.  Thus, $M$ has no direct summand isomorphic to $\Soc(M)$. Hence, as $\Soc(M)$ is finitely generated, $M$ is not strongly virtually regular. Therefore, we must have $k\leq 1$, and so $M\cong  R^n \oplus (\frac{R}{Rp})^{m}$ has the desired decomposition. This proves the necessity.

For the sufficiency, suppose $M\cong R^n \oplus (\frac{R}{Rp})^m$. Let $N$ be a finitely generated submodule of $M$. Without loss of generality, we may assume that $N \subseteq R^n \oplus (\frac{R}{Rp})^m.$ Thus $N\cong R^k \oplus (\frac{R}{Rp})^l$ for some nonnegative integers $k \leq n$ and $l \leq m.$ In this case clearly, $N$ is isomorphic to a direct summand of $R^n \oplus (\frac{R}{Rp})^m$, and so also isomorphic to  a direct summand of $M$. Thus $M$ is strongly virtually regular.
\end{proof}

\begin{prop}\label{prop:cvrovervaluationdomains} Let $R$ be a valuation domain with the unique maximal ideal $P$ and $M$ be a nonzero finitely presented  $R$-module. 

\begin{enumerate}

\item[(1)] If $M$ is completely virtually regular, then $P=Rp$ is principal.

\item[(2)] If $R$ is a DVR, then $M$ is completely virtually regular  if and only if $M\cong R^n \oplus (\frac{R}{Rp})^{m},$ where $n,\,m$ are nonnegative integers. 

\item[(3)]  If $R$ is not a DVR,  then $M$ is completely virtually regular if and only if  $M \cong (\frac{R}{Rp})^{m}$.

\end{enumerate}

\end{prop}

\begin{proof} $(1)$ Suppose $M$ is  completely virtually regular. Then $M$ is virtually regular. Since $M$ is nonzero, $M \neq T(M)$ or $T(M) \neq 0$. If $M \neq T(M)$, then $M$ has a direct summand isomorphic to $R$ by Proposition \ref{torsion-freecvssoverdomain}. Then $R$ is virtually regular by the assumption. Therefore $R$ is PID by Proposition \ref{prop:Rvirtuallyregular}(3), and so $P$ is principal. Now, if $T(M) \neq 0$, then $P$ is principal by Proposition \ref{prop:valuationadmitstorsionvr}. Hence $(1)$ follows.

$(2)$ Suppose $R$ is a DVR, and $M$ completely virtually regular.  Then $M$ is virtually regular, and so  $M \cong (\frac{R}{Rp})^{n_1} \oplus (\frac{R}{Rp^2})^{n_2} \oplus \cdots \oplus (\frac{R}{Rp^k})^{n_k}$ for nonnegative integers $\,k,\,n_1,\,n_2,\cdots ,n_k$ by Theorem \ref{thm:fpvirtuallyregularovervd}.  Assume that $k\geq 1$. By the assumption,  $\frac{R}{Rp^k}$ is  virtually regular as it is isomorphic to a submodule of $M$.   Then $Rp^k$ is a prime ideal of $R$ by Proposition \ref{lem:cyclicvirtuallyregular}, because $\frac{R}{Rp^k}$ is indecomposable. Therefore $k=1$, and so $M\cong R^n \oplus (\frac{R}{Rp})^{m}$. This proves the necessity. Since every  DVR is a Dedekind domain, sufficiency follows by Corollary \ref{cor:cvroverDD}.

$(3)$ Sufficiency is clear. For the necessity, suppose that $M$ is completely virtually regular. Let us show that, $M$ is torsion. Assume that $M$ is not torsion. Then $M$ has a submodule $A$ such that $A \cong R$. Thus $R$ is virtually regular, and so $R$ is a PID, and so a DVR by Proposition \ref{prop:Rvirtuallyregular}(3). This contradict with our assumption that $R$ is not a DVR. Therefore $M$ is torsion. Now, by similar arguments as in the proof of $(2)$, we obtain that $M\cong (\frac{R}{Rp})^{m}.$
\end{proof}

\vspace{0.5cm}

\newpage

%
%
%
%
%
%
%
%
%
%
%
%
%
%
%
%



\begin{table}[!ht]
\caption{The following table summarize the structure of  finitely presented (strongly) virtually regular and completely virtually regular modules according to the  results   given in Theorem \ref{thm:fpvirtuallyregularovervd}, Proposition \ref{prop:svrovervd} and Proposition \ref{prop:cvrovervaluationdomains} for a valuation domain $R$ with the unique maximal ideal $P$.}
\renewcommand{\arraystretch}{2.5} 
\setlength{\tabcolsep}{1em} 

\renewcommand{\arraystretch}{2.5} 
\setlength{\tabcolsep}{1em} 

\renewcommand{\arraystretch}{2} 
\setlength{\tabcolsep}{1em} 

	\centering
\begin{tabular}{|m{12 em} | c |m{8 em}|c|c|c|}
	\cline{2-4} \multicolumn{1}{c|}{} & $P$ is not principal& \multicolumn{2}{|c|}{  $P = Rp$ is principal} \\ [10pt] \hline 
	\vfill  \it{Virtually regular} \vfill & $R^n$ & \multicolumn{2}{|c|}{${\displaystyle R^n \oplus \left(\frac{R}{Rp}\right)^{n_1}  \oplus \cdots \oplus \left(\frac{R}{Rp^k}\right)^{n_k}}$} \\  \hline 
	\vfill  \it{Strongly virtually regular} \vfill & $R^n$ & \multicolumn{2}{|c|}{${\displaystyle R^n \oplus \left(\frac{R}{Rp}\right)^{m}}$}  \\   \hline
	\vfill \multirow{2}*{\it Completely virtually regular}   &  & \   $R$ is a DVR &  $R$ is not a DVR \\   \cline{3-4}  
	& 0 & \ \ \ $ {\displaystyle R^n \oplus  \left(\frac{R}{Rp}\right)^{m}}$  &    $ {\displaystyle \left(\frac{R}{Rp}\right)^{m}}$ \\ [8pt]  \hline 

\end{tabular}


\end{table}


\section{Abelian Groups}

In this section, we characterize finitely generated virtually regular and completely virtually regular modules over the ring of integers. 
Recall that, for a prime integer $p$, an abelian group $G$ is said to be \textit{$p$-group} if the order of every element of $G$ is a power of $p$.

\begin{prop}\label{finiteabelcvirtually}
A nonzero finite abelian $p$-group $M$ is virtually regular if and only if $$M \cong (\mathbb{Z}_{p})^{a_1} \oplus (\mathbb{Z}_{p^2})^{a_2} \oplus \cdots \oplus (\mathbb{Z}_{p^k})^{a_k}$$ for some positive integers $a_1, a_2,\cdots, a_k$.
\end{prop}

\begin{proof}
Let $M$ be a nonzero finite abelian $p$-group. By the Fundamental Theorem of Finitely Generated Abelian Groups (see \cite{fuchs}), we have $$M \cong \mathbb{Z}_{p^{n_1}} \oplus \mathbb{Z}_{p^{n_2}}\oplus \cdots \oplus \mathbb{Z}_{p^{n_k}} $$ where $n_1<n_2<\cdots <n_k$ are positive integers. Now, suppose $M$ is virtually regular. Clearly, for each $1 \leq t \leq n_k$, $\mathbb{Z}_{p^{n_k}}$ has a cyclic subgroup isomorphic $\mathbb{Z}_{p^{t}}$. Thus, by the assumption, $M$ has a direct summand isomorphic to  $\mathbb{Z}_{p^{t}}$. Therefore $M$ has the desired form:  $$M \cong (\mathbb{Z}_{p})^{a_1} \oplus (\mathbb{Z}_{p^2})^{a_2} \oplus \cdots \oplus (\mathbb{Z}_{p^k})^{a_k},$$and this proves the necessity.
Conversely, let $A$ be a nonzero cyclic subgroup of $M$. Since $A$ is a finite cyclic $p$-subgroup, $A \cong \mathbb{Z}_{p^s}$ for $s \in \mathbb{Z}^{+}$. As $p^kM=0,$ we must have $1 \leq s \leq k$. Clearly, by the assumption, $M$ has a direct summand isomorphic to $\Z_{p^s}.$  Therefore $A$ is isomorphic to a direct summand of $M$. Hence $M$ is virtually regular.
\end{proof}

Let $\Omega $ be the set of prime integers. It is well-known that $T(M) = \bigoplus\limits_{p \in \Omega} T_{p}(M),$ where $T_{p}(M)$ is the $p$-primary component of $T(M)$. The following lemma is clear and well-known, we include it for completeness.
\begin{lemma}\label{abeliantorsionpart}
Let $M$ be an abelian group. If $M = \bigoplus_{p \in \Omega} T_{p}(M)$ and $A \subseteq M$, then $A = \bigoplus_{p \in \Omega} \big(A \cap T_{p}(M)\big)$.
\end{lemma}

For torsion virtually regular modules, we have the following.
\begin{prop}\label{torsionabelTpm}
Let $M$ be a torsion abelian group. Then $M$ is virtually regular if and only if $T_{p}(M)$ is virtually regular for each prime $p$.
\end{prop}

\begin{proof}
To prove the necessity, suppose $M$ is virtually regular. Let  $A$ be a cyclic subgroup of $T_{p}(M)$, then $A \cong B \ds M$ by the assumption.  Let $M=B\oplus C$ for some $C$ of $M$. Since $A$ is a $p$-group, $B$ is also a $p$-group. Then by the modular law, we have $T_p(M)= B \oplus [T_p(M)\cap C]$. Thus, $A$ is isomorphic to the direct summand $B$ of $T_p(M)$. Therefore $T_p(M)$ is virtually regular. For the sufficiency, assume that $B$ is a cyclic subgroup of $M$.  Since $B$ is cyclic and torsion, $B=T_{p_1}(B)\oplus \cdots \oplus T_{p_n}(B)$ for some prime integers $p_1,\cdots ,p_n$. Now, $T_{p_i}(B)$ is a cyclic subgroup of $T_{p_i}(M)$ for each $i=1,\cdots ,n$. As $T_{p_i}(M)$ is virtually regular, $T_{p_i}(B) \cong A_i \ds T_{p_i}(M)$. Since $M=\oplus_{p\in \Omega}T_p(M)$, we have $A_1 \oplus \cdots \oplus A_n \ds M.$ Therefore $B\cong A_1 \oplus \cdots \oplus A_n \ds M$, and so $M$ is virtually regular.
\end{proof}

\begin{thm}\label{cvfag}
Let $M$ be a finitely generated abelian group. Then the following are equivalent.
 \begin{enumerate}
     \item $M$ is virtually regular.
     \item $T(M)$ is virtually regular.
     \item $T_{p}(M)$ is virtually regular for each prime $p$.
     \item $T_{p}(M) \cong (\mathbb{Z}_{p})^{a_1} \bigoplus (\mathbb{Z}_{p^2})^{a_2} \bigoplus \cdots \bigoplus (\mathbb{Z}_{p^k})^{a_k}$ for some positive integers $a_1, a_2,\cdots, a_k$ and each prime $p$.
 \end{enumerate}
\end{thm}
\begin{proof}

$(1) \iff (2)$ Since $M$ is finitely generated, $M \cong \Z^n \oplus T(M)$ for some nonnegative integer $n$.  $\Z^n$ is virtually regular by Proposition \ref{torsion-freecvssoverdomain}. Now the proof follows by  Corollary \ref{cor:abeltorsiontorsionfree}.




$(2) \iff (3)$ By Proposition \ref{torsionabelTpm}.

$(3) \iff (4)$ By Proposition \ref{finiteabelcvirtually}.
\end{proof}

Over the ring of integers, finitely generated virtually regular modules need not be virtually semisimple. For example, $\Z_2\oplus \Z_4$ is virtualy regular but it is not virtually semisimple. For cyclic abelian groups, we have the following.

\begin{lemma} Let $M$ be a cyclic abelian group. Then the following statements are equivalent.
\begin{enumerate}
  \item $M$ is virtually regular.
  \item $M$ is virtually semisimple.
  \item $M\cong \Z$ or $ M\cong \bigoplus_{i=1}^n \Z_{p_i}$, where $p_1, \cdots , p_n$ are distinct primes.
\end{enumerate}
\end{lemma}

\begin{proof} $(1) \Rightarrow (3)$ is clear by  Theorem \ref{cvfag}.

$(3) \Rightarrow (2)\Rightarrow (1)$  are  clear.

\end{proof}

\begin{corollary}\label{corollary:vss-cvr-svr} Let $M$ be a finitely generated abelian group. The following statements are equivalent.
\begin{enumerate}
  \item $M$ is virtually semisimple.
  \item $M$ is strongly virtually regular.
  \item $M$ is completely virtually regular.
  \item $T(M)$ is semisimple.
 \end{enumerate}
\end{corollary}

\begin{proof} $(1) \iff (2)$ For a Noetherian module the notions of virtually semisimple and strongly virtually regular are coincide. 

$(1) \iff (4)$ By \cite[Theorem 2.3]{svss}, and the fact that every nonzero prime ideal of $\Z$ is a maximal ideal.

$(3) \iff (4)$ By Corollary \ref{cor:cvroverDD}.

\end{proof}





\section*{Acknowledgement}
The authors would like to thank Prof. François Couchot for pointing out Corollary \ref{cor:propercyclicpureinjective}.

\end{document}